\documentclass[11pt,a4paper]{article}
\usepackage{amsmath,amssymb,amsthm}
\usepackage[margin=2.7cm]{geometry}
\usepackage{booktabs}
\usepackage{setspace,tikz,graphicx}
\usepackage[colorlinks=true,linkcolor=blue,citecolor=blue,urlcolor=blue]{hyperref}
\usepackage{soul}
\theoremstyle{plain}
\newtheorem{theorem}{Theorem}[section]
\newtheorem{lemma}[theorem]{Lemma}
\newtheorem{proposition}[theorem]{Proposition}
\newtheorem{corollary}[theorem]{Corollary}
\newtheorem{conjecture}[theorem]{Conjecture}
\theoremstyle{definition}

\newtheorem{question}[theorem]{Question}
\theoremstyle{remark}

\newcommand{\dt}{\mathrm{td}}
\newcommand{\FTD}{\mathrm{FTD}}

\newcommand{\Hw}{\mathbb{H}}

\newcommand{\bin}{\binom{[4]}{2}}

\title{Every cubic graph has four total dominating sets\\
covering each vertex at most twice}

\author{%
Hyemin Kwon\thanks{%
Department of Mathematics Education, 
Seoul National University, 
Seoul, Republic of Korea. 
\texttt{khmin1121@snu.ac.kr}
}
\and
Boram Park\thanks{%
Department of Mathematics Education, 
Seoul National University, 
Seoul, Republic of Korea. 
\texttt{borampark@snu.ac.kr}}
\and
Yoshio Sano\thanks{%
Department of Computer Science, 
Institute of Systems and Information Engineering, 
University of Tsukuba, 
Tsukuba, Ibaraki, Japan.
\texttt{sano@cs.tsukuba.ac.jp}}}
\date{}

\begin{document}
\maketitle

\begin{abstract}
In 2018, Goddard and Henning conjectured that every cubic graph has four total dominating sets covering each vertex at most twice, and asked the analogous question for transversals of $3$-uniform $3$-regular hypergraphs.
We prove the conjecture, which also answers the question in the affirmative.


\end{abstract}

\section{Introduction}

All graphs in this paper are finite and simple.
For a graph $G$, let $N_G(v)$ be the open neighborhood of a vertex $v$, and $N_G(X)=\cup_{v\in X} N_G(v)$ for $X\subseteq V(G)$.
A set $T\subseteq V(G)$ is a \emph{total dominating set} of $G$ if every vertex of $G$ has a neighbor in $T$.
The \emph{total domatic number} $\dt(G)$ of $G$ is the largest number of pairwise disjoint total dominating sets of $G$.
It was introduced by Cockayne, Dawes, and Hedetniemi~\cite{CDH}, and it is also known as the \emph{coupon coloring number}~\cite{CKTV}, the largest $k$ for which $V(G)$ has a coloring with $k$ colors in which every vertex sees all $k$ colors in its neighborhood.

Zelinka~\cite{Zel} showed that there are graphs of arbitrarily large minimum degree with total domatic number one.
Regular graphs behave much better.
Chen, Kim, Tait, and Verstra\"ete~\cite{CKTV} showed that a $k$-regular graph has total domatic number at least $(1-o(1))k/\ln k$ as $k\to\infty$, and a theorem of Thomassen~\cite{Tho} on $k$-uniform $k$-regular hypergraphs implies that $\dt(G)\ge2$ for every $k$-regular graph $G$ with $k\ge4$.
For cubic graphs this fails.
The Heawood graph $\Hw$ given in Figure~\ref{fig:heawood} has $\dt(\Hw)=1$, and there are infinitely many connected cubic graphs with total domatic number one (see~\cite{Akb18,Akb24,CKP,DHH} for the total domatic number of cubic graphs).


\begin{figure}[b!]
\centering
\begin{tikzpicture}[scale=0.9,
 pt/.style={circle,fill,inner sep=1.6pt},
 ln/.style={circle,draw,fill=black,inner sep=1.6pt},
 lab/.style={font=\small}]
\foreach \i/\lbl in {0/p_1,1/\ell_{123},2/p_2,3/\ell_{246},4/p_4,5/\ell_{145},6/p_5,7/\ell_{257},8/p_7,9/\ell_{347},10/p_3,11/\ell_{356},12/p_6,13/\ell_{167}}{
 \pgfmathsetmacro{\ang}{90-\i*360/14}
 \ifodd\i
   \node[ln] (h\i) at (\ang:2.4) {};
 \else
   \node[pt] (h\i) at (\ang:2.4) {};
 \fi
 \node[lab] at (\ang:2.9) {$\lbl$};
}
\foreach \i in {0,...,13}{
 \pgfmathtruncatemacro{\j}{mod(\i+1,14)}
 \draw (h\i) -- (h\j);
}
\foreach \i in {0,2,4,6,8,10,12}{
 \pgfmathtruncatemacro{\j}{mod(\i+5,14)}
 \draw (h\i) -- (h\j);
}
\foreach \i in {1,3,5,7,9,11,13}{\node[ln] at (h\i) {};}
\end{tikzpicture}
\caption{The Heawood graph $\Hw$ (see Section~\ref{sec:heawood}).}
\label{fig:heawood}
\end{figure}
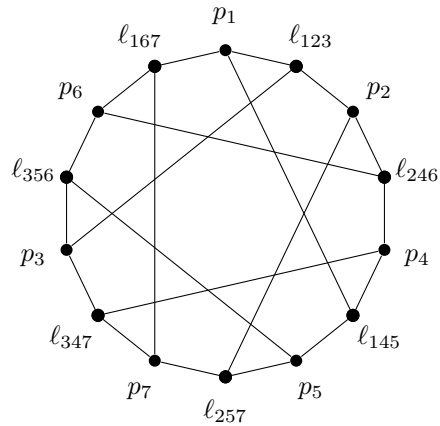

Goddard and Henning~\cite{GH} proposed a fractional version.
A \emph{thoroughly dispersed family} of $G$ is a family $\mathcal F$ of (not necessarily distinct) total dominating sets of $G$.
Let $r_{\mathcal F}$ be the largest number of members of $\mathcal F$ that contain a common vertex.
The \emph{fractional total domatic number} of $G$ is
\[
  \FTD(G)=\sup_{\mathcal F}\ |\mathcal F|/r_{\mathcal F}.
\]
Clearly $\dt(G)\le\FTD(G)$.
Henning and Yeo~\cite{HY17} proved that $\FTD(G)\ge2$ for every connected cubic graph $G$. Note that the ratio $2$ cannot be increased in general, since every total dominating set of $K_4$ has at least two of its four vertices, so $\FTD(K_4)=2$.
Goddard and Henning also posed the following stronger conjecture.

\begin{conjecture}[{\cite[Conjecture 38]{GH}}]
\label{conj:GH}
If $G$ is a connected cubic graph, then $G$ has a thoroughly dispersed family of four sets such that every vertex is in at most two of these.
\end{conjecture}

More generally, they asked the following. A \emph{transversal} of a hypergraph $H$ is a set of vertices that intersects every hyperedge of $H$. 
The open neighborhoods of a cubic graph, counted with multiplicity, form a $3$-uniform $3$-regular hypergraph.
The transversals of this hypergraph are exactly the total dominating sets of the original graph.
Thus, the following question generalizes Conjecture~\ref{conj:GH}.

\begin{question}[{\cite[Question 39]{GH}}]\label{q:GH}
Does every $3$-uniform $3$-regular hypergraph have four transversals with each vertex in at most two?
\end{question}

In the final section of~\cite{GH} the authors listed Conjecture~\ref{conj:GH} as one of the two open problems they found most frustrating.
The conjecture is easy when $G$ is a cubic graph with $\dt(G)\ge2$, since one can take two disjoint total dominating sets and use each of them twice.
So the difficulty lies in the cubic graphs with total domatic number one.
We prove the conjecture.

\begin{theorem}\label{thm:main}
Every cubic graph has four total dominating sets such that every vertex belongs to exactly two of them.
\end{theorem}

Theorem~\ref{thm:main} also answers Question~\ref{q:GH} in the affirmative.

\begin{corollary}\label{cor:hyper}
Every $3$-uniform $3$-regular hypergraph has four transversals such that every vertex belongs to exactly two of them.
\end{corollary}

Theorem~\ref{thm:main} yields the lower bound $\FTD(G)\ge2$ of Henning and Yeo~\cite{HY17}, and the bound can always be realized by a family of four total dominating sets. 


\section{Preliminaries}

\subsection{Basic observations}\label{sec:prelim}

Throughout the paper, $[4]=\{1,2,3,4\}$ and $\bin$ is the set of the six $2$-element subsets of $[4]$, which we call \emph{colors}.
For a color $\pi$, let $\overline{\pi}=[4]\setminus\pi$ be the complementary color.

A \emph{$(4{:}2)$-coloring} of a graph $G$ without isolated vertices is a map $c:V(G)\to\bin$ such that
\[
  \bigcup_{u\in N_G(v)}c(u)=[4]\qquad\text{for every } v\in V(G).
\]
We then say that $v$ \emph{sees all four colors}.

\begin{proposition}\label{prop:code}
Let $G$ be a graph without isolated vertices. The following are equivalent.
\begin{enumerate}
\item[\textup{(i)}] $G$ has four total dominating sets such that every vertex is in at most two of them.
\item[\textup{(ii)}] $G$ has four total dominating sets such that every vertex is in exactly two of them.
\item[\textup{(iii)}] $G$ has a $(4{:}2)$-coloring.
\end{enumerate}
\end{proposition}

\begin{proof}
Clearly (ii) implies (i).

To see that (i) implies (iii), let $T_1,\dots,T_4$ be as in (i). For each $v\in V(G)$, let $I(v)=\{i:v\in T_i\}$, and choose $c(v)\in\bin$ with $c(v)\supseteq I(v)$, which is possible since $|I(v)|\le2$. For $v\in V(G)$ and $i\in[4]$, the vertex $v$ has a neighbor $u\in T_i$, and then $i\in I(u)\subseteq c(u)$. Hence $c$ is a $(4{:}2)$-coloring.

To see that (iii) implies (ii), let $c$ be a $(4{:}2)$-coloring  of $G$, and let $T_i=\{v:i\in c(v)\}$ for $i\in[4]$. Every vertex lies in exactly two of these sets. For $v\in V(G)$ and $i\in[4]$, the vertex $v$ has a neighbor $u$ with $i\in c(u)$, that is, $u\in T_i$. Hence each $T_i$ is a total dominating set.
\end{proof}

\begin{proposition}\label{prop:easy}
If $G$ is a graph with $\dt(G)\ge2$, then $G$ has a $(4{:}2)$-coloring.
\end{proposition}

\begin{proof}
Let $T$ and $V(G)\setminus T$ be total dominating sets.
Color the vertices of $T$ with $\{1,2\}$ and the other vertices with $\{3,4\}$.
\end{proof}

The \emph{bipartite double cover} $B(G)$ of a graph $G$ has vertex set $V(G)\times\{1,2\}$, and $(u,1)$ is adjacent to $(v,2)$ whenever $uv\in E(G)$. If $G$ is cubic, then so is $B(G)$, and it is easy to check that $B(G)$ is connected if $G$ is connected and non-bipartite.

\begin{proposition}\label{prop:double}
A graph $G$ without isolated vertices has a $(4{:}2)$-coloring if and only if $B(G)$ has one.
\end{proposition}

\begin{proof}
Write $u_i=(u,i)$.
Let $c$ be a $(4{:}2)$-coloring of $G$, and let $\tilde c(u_i)=c(u)$.
Since $N_{B(G)}(u_1)=\{v_2:v\in N_G(u)\}$, the colors seen by $u_1$ are the colors seen by $u$ in $G$, and the same holds for $u_2$.
So $\tilde c$ is a $(4{:}2)$-coloring of $B(G)$.
Conversely, let $\tilde c$ be a $(4{:}2)$-coloring of $B(G)$, and let $c(u)=\tilde c(u_2)$.
Then the colors seen by $u$ in $G$ are the colors seen by $u_1$ in $B(G)$.
So $c$ is a $(4{:}2)$-coloring of $G$.
\end{proof}


A connected bipartite graph with a perfect matching is \emph{$1$-extendable} if every edge lies in a perfect matching.
By K\H{o}nig's edge coloring theorem, a cubic bipartite graph is the union of three perfect matchings, so every connected cubic bipartite graph is $1$-extendable.
We use the following structure theorem of McCuaig~\cite{McC00}.

\begin{theorem}[{\cite[Theorems~22 and 27]{McC00}}]\label{thm:mccuaig}
Let $G$ be a connected $1$-extendable bipartite graph with minimum degree at least $3$, and assume that $\dt(G)=1$.
Then there is a set $A\subsetneq V(G)$ such that $G[A]\cong\Hw-x$ for a vertex $x$ of the Heawood graph $\Hw$.
\end{theorem}



McCuaig~\cite{McC00} states his results in terms of unbalanced graphs.
Since $G$ is connected and $1$-extendable, $\dt(G)=1$ implies that $G$ is unbalanced by~\cite[Theorem~3]{RRT}; see also~\cite[Theorem~13]{McC04}.
If $G$ is a brace, meaning that every pair of disjoint edges belongs to a common perfect matching, then $G\cong \Hw$ by~\cite[Theorem~22]{McC00}.
Otherwise, \cite[Theorem~27]{McC00} gives a reduction isomorphic to $\Hw$.
Deleting the new vertex of that reduction leaves an induced copy of $\Hw-x$ in $G$.

\subsection{A $(4{:}2)$-coloring of the Heawood graph}\label{sec:heawood}

The \emph{incidence graph} of a hypergraph is the bipartite graph whose vertices are the vertices and the edges of the hypergraph, with a vertex adjacent to the edges containing it.
The \emph{Fano plane} $F_7=(P,\mathcal{L})$ is the $3$-uniform $3$-regular hypergraph with vertex set $P=\{p_1,\dots,p_7\}$ and edge set
\[
  \mathcal{L}=\{\ell_{123},\,\ell_{145},\,\ell_{167},\,\ell_{246},\,\ell_{257},\,\ell_{347},\,\ell_{356}\},
\]
where we write $\ell_{ijk}$ for $\{p_i,p_j,p_k\}$.
Its vertices are called \emph{points} and its edges are called \emph{lines}.
Every two distinct points lie on exactly one common line, and the automorphism group of $F_7$ acts transitively on points.
The Heawood graph $\Hw$ is the incidence graph of $F_7$, see Figure~\ref{fig:heawood}.
It is well known that $\dt(\Hw)=1$, since $F_7$ has no $2$-coloring without a monochromatic line.

The Fano plane is \emph{self-dual}, that is, there is a bijection $\delta\colon\mathcal{L}\to P$ such that, for each point $q$, the images under $\delta$ of the three lines containing $q$ are the three points of a line, which we denote by $\ell_q$.
One may take
\[
  \delta(\ell_{123})=p_1,\ \delta(\ell_{145})=p_6,\ \delta(\ell_{167})=p_7,\ \delta(\ell_{246})=p_2,\ \delta(\ell_{257})=p_3,\ \delta(\ell_{347})=p_4,\ \delta(\ell_{356})=p_5.
\]
Composing $\delta$ with an automorphism of $F_7$ gives another such bijection, so for any line $y$ and any point $x$ we may choose $\delta$ with $\delta(y)=x$.
Moreover, the map sending each line $\ell$ to $\delta(\ell)$ and each point $q$ to $\ell_q$ is an automorphism of $\Hw$ that exchanges points and lines, since $q\in\ell$ if and only if $\delta(\ell)\in\ell_q$.

\begin{lemma}\label{lem:fano}
Let $F_7=(P,\mathcal{L})$ be the Fano plane and let $p \in P$. Then there exists a map $c \colon P \setminus \{p\} \to \binom{[4]}{2}$ such that
\[ \bigcup_{q \in \ell \setminus \{p\}} c(q) = [4]  \qquad \text{for every } \ell \in \mathcal{L}.\]
\end{lemma}

\begin{proof}
Since the automorphism group of the Fano plane acts transitively on points, we may assume $p=p_7$. The coloring
\[
  c(p_1)=\{1,2\},\ c(p_2)=\{1,3\},\ c(p_3)=\{1,4\},\ c(p_4)=\{2,3\},\ c(p_5)=\{2,4\},\ c(p_6)=\{3,4\}
\]
has the required property, as is easily checked.
\end{proof}

\begin{lemma}\label{lem:strong}
Let $x$ be a vertex of the Heawood graph $\Hw$ and let $\pi\in\bin$.
Then $\Hw-x$ has a $(4{:}2)$-coloring $c$ with $c(w)=\pi$ for every $w\in N_{\Hw}(x)$.
\end{lemma}

\begin{proof}
By self-duality, we may assume that $x$ is a point.
Color each point $q\ne x$ by the map of Lemma~\ref{lem:fano} with $p=x$, the lines containing $x$ with $\pi$, and the other lines with $\overline{\pi}$.
Each line $\ell$ is adjacent in $\Hw-x$ to all points of $\ell\setminus\{x\}$, whose colors cover $[4]$ by Lemma~\ref{lem:fano}.
Each point $q\ne x$ is adjacent in $\Hw-x$ to the line containing $x$ and $q$, which has color $\pi$, and to the two other lines containing $q$, which have color $\overline{\pi}$.
Since $\pi\cup\overline{\pi}=[4]$, this is a $(4{:}2)$-coloring of $\Hw-x$, and every neighbor of $x$ has color $\pi$ by construction.
\end{proof}

\begin{lemma}\label{lem:minus}
Let $e=xy$ be an edge of the Heawood graph $\Hw$ and let $\pi_1,\pi_2\in\bin$.
Then $\Hw-e$ has a $(4{:}2)$-coloring $c$ with $c(x)=\pi_1$ and $c(y)=\pi_2$.
\end{lemma}

\begin{proof}
By self-duality, we may assume that $x$ is a point and $y$ is a line containing $x$.
Choose $\delta$ as above with $\delta(y)=x$, and let $c_0$ be the map of Lemma~\ref{lem:fano} with $p=x$.
Color $x$ with $\pi_1$, $y$ with $\pi_2$, each point $q\ne x$ with $c_0(q)$, and each line $\ell\ne y$ with $c_0(\delta(\ell))$.

Each line $\ell$ is adjacent in $\Hw-e$ to all points of $\ell\setminus\{x\}$, whose colors cover $[4]$ by Lemma~\ref{lem:fano}.
Each point $q$ is adjacent in $\Hw-e$ to all lines containing $q$ other than $y$.
Recall that the images under $\delta$ of the three lines containing $q$ are the points of $\ell_q$.
Since $\delta(y)=x$, the images of the lines containing $q$ other than $y$ are exactly the points of $\ell_q\setminus\{x\}$, so their colors cover $[4]$ by Lemma~\ref{lem:fano}.
Hence this is a $(4{:}2)$-coloring of $\Hw-e$ with the required colors on $x$ and $y$.
\end{proof}

For any edge $e$ of $\Hw$, every $(4{:}2)$-coloring of $\Hw-e$ is also a $(4{:}2)$-coloring of $\Hw$, so Lemma~\ref{lem:minus} yields the following corollary.

\begin{corollary}\label{cor:heawood}
The Heawood graph $\Hw$ has a $(4{:}2)$-coloring.
\end{corollary}

\section{Proof of Theorem~\ref{thm:main}}\label{sec:proof}

\begin{theorem}\label{thm:bip}
Every connected cubic bipartite graph has a $(4{:}2)$-coloring.
\end{theorem}

\begin{proof}
Suppose not, and let $G$ be a counterexample minimizing $|V(G)|$.
By Proposition~\ref{prop:easy}, $\dt(G)=1$. By Corollary~\ref{cor:heawood}, $G$ is not isomorphic to $\Hw$.

Since $G$ is connected, cubic and bipartite, it is $1$-extendable, and Theorem~\ref{thm:mccuaig} gives a set $A\subsetneq V(G)$ with $G[A]\cong\Hw-x$ for some vertex $x$ of $\Hw$.
Identify $G[A]$ with $\Hw-x$.
Then the vertices of degree two in $G[A]$ are exactly the three neighbors $t_1$, $t_2$, $t_3$ of $x$ in $\Hw$, and they are in the same part of the bipartition of $G$.
Since $G$ is cubic, exactly three edges join $A$ to $V(G)\setminus A$, namely $t_is_i$ at each $t_i$.
If $s_1=s_2=s_3$, then $V(G)\setminus A=\{s_1\}$ since $G$ is connected, and so $G\cong\Hw$.
Since $G\not\cong\Hw$, by symmetry we may assume that either $s_1,s_2,s_3$ are distinct or $s_1=s_2\ne s_3$.

We define a set $A^*\supseteq A$, a connected bipartite graph $Q$, and a graph $G'$ as follows:
\begin{itemize}
\item[(a)] If $s_1,s_2,s_3$ are distinct, then let $A^*=A$, and let $Q$ be a single vertex $x^*$.
Let $G'$ be obtained from the disjoint union of $G-A^*$ and $Q$ by adding the edges $x^*s_1$, $x^*s_2$ and $x^*s_3$.

\item[(b)] If $s_1=s_2$, then $N_G(s_1)=\{t_1,t_2,x'\}$ for some $x'\notin A$.
Let $A^*=A\cup\{s_1\}$, and let $Q=K_{3,3}-pq$ for an edge $pq$ of $K_{3,3}$.
Let $G'$ be obtained from the disjoint union of $G-A^*$ and $Q$ by adding the edges $px'$ and $qs_3$.
\end{itemize}
In case~(b), the vertices $x'$ and $s_3$ lie in opposite parts of the bipartition of $G$, and so do $p$ and $q$ in $K_{3,3}-pq$.
Hence, in both cases, $G'$ is simple, cubic and connected, and it is bipartite with each vertex of $Q$ joined to $G-A^*$ placed opposite its neighbors in $G-A^*$.
Moreover, $|V(G')|<|V(G)|$ in both cases.
By the minimality of $G$, the graph $G'$ has a $(4{:}2)$-coloring $c'$.

Next, we color $G[A^*]$ according to the colors of $c'$ on $Q$.
Note that $G[A^*]$ is isomorphic to $\Hw-x$ in case~(a) and to $\Hw-e$  in case~(b), where the endpoints of $e$ correspond to $s_1$ and $t_3$.
In case~(a), Lemma~\ref{lem:strong} gives a $(4{:}2)$-coloring $c$ of $G[A^*]$ with $c(t_i)=c'(x^*)$ for each $i$.
In case~(b), Lemma~\ref{lem:minus} gives a $(4{:}2)$-coloring $c$ of $G[A^*]$ with $c(s_1)=c'(p)$ and $c(t_3)=c'(q)$.

Use $c$ on $A^*$ and $c'$ on $V(G)\setminus A^*$.
Every vertex of $A^*$ sees all four colors by the choice of $c$.
For a vertex of $G-A^*$, the only possible change of its neighborhood is that some neighbors in $Q$ are replaced by neighbors in $A^*$ of the same colors.
Hence this gives a $(4{:}2)$-coloring of $G$, which is a contradiction.
\end{proof}

\begin{proof}[Proof of Theorem~\ref{thm:main}]
A $(4{:}2)$-coloring can be chosen separately on each component, so let $G$ be connected.
If $G$ is bipartite, apply Theorem~\ref{thm:bip}.
Otherwise its bipartite double cover $B(G)$ is connected and cubic, so Theorem~\ref{thm:bip} gives a $(4{:}2)$-coloring of $B(G)$, and
Proposition~\ref{prop:double} gives one of $G$.
Proposition~\ref{prop:code} now yields four total dominating sets such that every vertex lies in exactly two of them.
\end{proof}

We now prove Corollary~\ref{cor:hyper}.

\begin{proof}[Proof of Corollary~\ref{cor:hyper}]
Let $H$ be a $3$-uniform $3$-regular hypergraph, where repeated edges are allowed.
Let $I(H)$ be the incidence graph of $H$.
Then $I(H)$ is a simple cubic bipartite graph, so by Theorem~\ref{thm:main} and Proposition~\ref{prop:code} it has a $(4{:}2)$-coloring $c$.
Then $\bigcup_{v\in e}c(v)=[4]$ for every edge $e$ of $H$.
Hence, for each $i\in[4]$, the set $T_i=\{v\in V(H): i\in c(v)\}$ is a transversal of $H$, and every vertex of $H$ lies in exactly two of $T_1,\dots,T_4$.
\end{proof}

\section*{Acknowledgments}
Boram Park was supported by the National Research Foundation of Korea (NRF) grant funded by the Korea government (MSIT) (No.~RS-2025-00523206) and by the New Faculty Startup Fund from Seoul National University.
Yoshio Sano was supported by JSPS KAKENHI Grant Number JP23K03194 and MEXT Supporting Pioneering Research through AI for 1,000 Discovery challenges Program (SPReAD) Japan Grant Number JPMXP1726292712.

\bigskip
\noindent\textbf{Declaration of generative AI use.}
\par\smallskip
\noindent  
Anthropic Claude Opus 5 and OpenAI GPT-6 Astra were
used to assist with preparing an initial draft. 
The authors subsequently revised the manuscript and take full responsibility for its final content.

\end{document}